\documentclass[12pt]{amsart}
\usepackage{amssymb, amscd}

\newtheorem{Lemma}             {Lemma}
\newtheorem{Corollary}  [Lemma]{Corollary}
\newtheorem{Proposition}[Lemma]{Proposition}
\newtheorem{Theorem}    [Lemma]{Theorem}

\newcommand{\ov}{\overline}
\newcommand{\op}{\operatorname}
\newcommand{\Mod}[1]{\ (\mathrm{mod}\ #1)}

\title[F-S indicators, cyclic defect]{Frobenius-Schur indicators of characters in blocks with cyclic defect}

\author{John C. Murray}
\address{Department of Mathematics and Statistics\\National University of Ireland Maynooth\\IRELAND}
\email{John.Murray@maths.nuim.ie}
\date{\today}
\begin{document}

\begin{abstract}Let $p$ be an odd prime and let $B$ be a $p$-block of a finite group which has cyclic defect groups. We show that all exceptional characters in $B$ have the same Frobenius-Schur indicators. Moreover the common indicator can be computed, using the canonical character of $B$. We also investigate the Frobenius-Schur indicators of the non-exceptional characters in $B$.

For a finite group which has cyclic Sylow $p$-subgroups, we show that the number of irreducible characters with Frobenius-Schur indicator $-1$ is greater than or equal to the number of conjugacy classes of weakly real $p$-elements in $G$. 
\end{abstract}
\maketitle
\thispagestyle{empty}

\section{Introduction and preliminary results}

The Frobenius-Schur (F-S) indicator of an ordinary character $\chi$ of a finite group $G$ is
$$
\epsilon(\chi):=\frac{1}{|G|}\sum_{g\in G}\chi(g^2).
$$
If $\chi$ is irreducible then $\epsilon(\chi)=0,\pm1$. Moreover $\epsilon(\chi)\ne0$ if and only if $\chi$ is real-valued.

R. Brauer showed how to partition the irreducible characters of $G$ into $p$-blocks, for each prime $p$. Each $p$-block has an associated defect group, which is a $p$-subgroup of $G$, unique up to $G$-conjugacy, which determines much of the structure of the block. If the defect group is trivial, the block contains a unique irreducible character. In the next most complicated case, E. Dade \cite{D} determined the structure of a block which has a cyclic defect group and defined the Brauer tree of the block.

Recall that a $p$-block is said to be real if it contains the complex conjugates of its characters. We wish to determine the F-S indicators of the irreducible characters in a real $p$-block which has a cyclic defect group. In \cite[Theorem 1.6]{M2} we dealt with the case $p=2$; there are six possible indicator patterns, and the {\em extended} defect group of the block determines which occurs. In this paper we consider the case $p\ne2$.

R. Gow showed \cite[5.1]{G} that a real $p$-block has a real irreducible character, if $p=2$. This is false for $p\ne2$, as was first noticed by H. Blau in the early 1980's, in response to a question posed by Gow. His example was for $p=5$ and $G=6.S_6$ ({\em Atlas} notation). G. Navarro has recently found a solvable example with $p=3$ and $G=\op{SmallGroup}(144,131)$ ({\em GAP} notation). We give examples for blocks with cyclic defect below.

Now let $B$ be a real $p$-block which has a cyclic defect group $D$. The inertial index of $B$ is a certain divisor $e$ of $p-1$. Dade showed that $B$ has $e$ irreducible Brauer characters and $e+\frac{|D|-1}{e}$ ordinary irreducible characters. The latter he divided into $\frac{|D|-1}{e}$ exceptional characters and $e$ non-exceptional characters.

Suppose that $\frac{|D|-1}{e}=1$ (which can only occur when $|D|=p$). Then the choice of exceptional character is arbitrary, and the convention in \cite{F} is to regard $B$ as having no exceptional characters. However, we will see that in this event $B$ has real irreducible characters, all of which have the same F-S indicators. So our convention is to assume that $B$ has a real exceptional character.

The Brauer tree of $B$ is a planar graph which describes the decomposition matrix of $B$. There is one exceptional vertex, representating all the exceptional characters, and one vertex for each of the non-exceptional characters. Two vertices are connected by an edge if their characters share a modular constituent.

J. Green \cite{Gr} showed that all real objects in the Brauer tree lie on a line segment, now called the real-stem of $B$. The exceptional vertex belongs to the real-stem (see Lemma \ref{L:real_stem} below). So it divides the real non-exceptional vertices into two, possibly empty, subsets. We find it convenient to refer to the corresponding real non-exceptional characters as being on the left or the right of the exceptional vertex. Here is our main theorem:

\begin{Theorem}\label{T:main}
Let $p$ be an odd prime and let $B$ be a real $p$-block which has a cyclic defect group. Then
\begin{itemize}
 \item[(i)] All exceptional characters in $B$ have the same F-S indicators.
 \item[(ii)] On each side of the exceptional vertex, the real non-exceptional characters have the same F-S indicators.
 \item[(iii)] If $B$ has a real exceptional character then all real irreducible characters in $B$ have the same F-S indicators.
 \item[(iv)] Suppose that $B$ has no real exceptional characters, and that there are an odd number of non-exceptional vertices on each side of the exceptional vertex. Then the real non-exceptional characters have F-S indicator $+1$ on one side of the exceptional vertex and $-1$ on the the other side.
\end{itemize}
\end{Theorem}

Note that (i) is not a consequence of Galois conjugacy, as there are at least two Galois conjugacy classes of exceptional characters, when $|D|>p$. 

In Proposition \ref{P:B0} we show that the F-S indicators of the exceptional characters in $B$ agree with those of the Brauer corresponding block in the normalizer of a defect group. In Theorem \ref{T:canonical} we compute this common indicator using the `canonical character' of $B$.

Next recall that an element of $G$ is said to be weakly real if it is conjugate to its inverse in $G$, but it is not inverted by any involution in $G$. Here is an application of Theorem \ref{T:main} whose statement does not refer to blocks or to modular representation theory:

\begin{Theorem}\label{T:symplectic}
Let $p$ be an odd prime and let $G$ be a finite group which has cyclic Sylow $p$-subgroups. Then the number of irreducible characters of\/ $G$ with F-S indicator $-1$ is greater than or equal to the number of conjugacy classes of weakly real $p$-elements in $G$. 
\end{Theorem}

We use the notation and results of \cite{NT} for group representation theory, and use \cite{D} and \cite[VII]{F} for notation specific to blocks with cyclic defect. When referring to the character tables of a finite simple group we use the conventions of the {\em ATLAS} \cite{A}. For other character tables, we use the notation of the computer algebra system {\em GAP} \cite{GAP}. 

\section{Examples}\label{S:examples}

We begin with a number of examples which illustrate the possible patterns of F-S indicators in a block which has a cyclic defect group. Throughout $G$ is a finite group and $B$ is a real $p$-block of $G$ which has a cyclic defect group $D$. Also $N_0$ is the normalizer in $G$ of the unique order $p$ subgroup of $D$ and $B_0$ is the Brauer correspondent of $B$ in $N_0$.  

\medskip
{\bf Example 1:} There are many blocks with cyclic defect group whose irreducible characters all have the same F-S indicators. For blocks with all indicators $+1$, choose $n\geq2$, a prime $p$ with $n/2\leq p\leq n$ and any $p$-block of the symmetric group $S_n$. There are numerous blocks with all indicators $-1$ among the faithful $p$-blocks of the double cover $2.A_n$ of an alternating group, with $n/2\leq p\leq n$ e.g. the four faithful irreducible characters of $2.A_5$ have F-S indicator $-1$ and constitute a $5$-block with a cyclic defect group.

\medskip
{\bf Example 2:} If $e$ is odd then $B$ has a real non-exceptional character. Now it follows from \cite[Part 2 of Theorem 1 \& Corollary 1.9]{D} that $B$ has a Galois conjugacy class consisting of $\frac{p-1}{e}$ exceptional characters. So $B$ has a real exceptional character if $\frac{p-1}{e}$ is odd. Thus $B$ always has a real irreducible character if $p\equiv3\Mod4$.

When $e$ is even and $p\equiv1\Mod4$, $B$ may have no real irreducible characters. For example SmallGroup$(80,29)=\langle a,b\mid a^{20},a^{10}=b^4,a^b=a^7\rangle$ has such a block, for $p=5$. It consists of the four irreducible characters lying over the non-trivial irreducible character of $\langle a^{10}\rangle$. Here is its character table. The first two rows indicate the 2 and 5 parts of the class centralizers. The third row labels the classes by their element orders:
$$
\begin{array}{lrrrrrrrrrrrrrrr}
 &2&4&4&3&3&4&4&2&3&3&3&3&2&2&2\\
 &5&1&1&.&1&.&.&1&.&.&.&.&1&1&1\\
 \\
 &&1a&2a&2b&4a&4b&4c&5a&8a&8b&8c&8d&10a&20a&20b\\
 \\
X.9&&2&-2&.&.&2i&-2i&2&.&.&.&.&-2&.&.\\
X.10&&2&-2&.&.&-2i&2i&2&.&.&.&.&-2&.&.\\
X.13&&4&-4&.&.&.&.&-1&.&.&.&.&1&\sqrt{-5}&-\sqrt{-5}\\
X.14&&4&-4&.&.&.&.&-1&.&.&.&.&1&-\sqrt{-5}&\sqrt{-5}\\
\end{array}
$$
Note that SmallGroup$(80,29)$ has Sylow $2$-subgroups isomorphic to SmallGroup$(16,6)=\langle s,t \mid s^8,t^2, s^t=s^5\rangle$. This 2-group is sometimes denoted $M_4(2)$.

\medskip
{\bf Example 3:} $B$ may have a real non-exceptional character but no real exceptional characters. For example SmallGroup$(60,7)=\langle a,b\mid a^{15},b^4,a^b=a^2\rangle$ has such a block, for $p=5$. It consists of the four irreducible characters lying over a non-trivial irreducible character of $\langle a^5\rangle$. This is also an example of part (iv) of Theorem \ref{T:main}; the non-exceptional characters $X.5$ and $X.6$ have F-S indicators $-1$ and $+1$, respectively. Here is the table of character values, with $\alpha=(1+\sqrt{-15})/2$:
$$
\begin{array}{lrrrrrrrrrr}
 &2&2&2&1&2&2&.&1&.&.\\
 &3&1&1&1&.&.&1&1&1&1\\
 &5&1&.&1&.&.&1&.&1&1\\
 \\
 &&1a&2a&3a&4a&4b&5a&6a&15a&15b\\
 \\
X.5&&2&-2&-1&.&.&2&1&-1&-1\\
X.6&&2&2&-1&.&.&2&-1&-1&-1\\
X.8&&4&.&-2&.&.&-1&.&\alpha&\ov\alpha\\
X.9&&4&.&-2&.&.&-1&.&\ov\alpha&\alpha\\
\end{array}
$$

\medskip
{\bf Example 4:} There is no apparent relationship between the F-S indicators of the non-exceptional characters in $B$ and in $B_0$. For example, let $B$ be the $5$-block $2.A_8$ with $\op{Irr}(B)=\{\chi_{15},\chi_{19},\chi_{21},\chi_{22}\}$. Then the two non-exceptional characters $\chi_{15}$ and $\chi_{19}$ have F-S indicator $+1$ and $-1$, respectively. However $B_0$ is a real block which has no real irreducible characters.

The character table of $B$ can be found on p22 of {\em The Atlas}. Now $N_0$ is isomorphic to SmallGroup$(120,7)=\langle a,b\mid a^{15},b^8,a^b=a^2\rangle$. Here is the table of character values of its $5$-block $B_0$. Again $\alpha=(1+\sqrt{-15})/2$. In order to save space, we have omitted $4$ columns of zero values for the four classes of elements of order $8$:
$$
\begin{array}{lrrrrrrrrrrrrrrrrrrr}
 &2&3&3&2&3&3&1&2&1&2&2&1&1&1&1\\
 &3&1&1&1&1&1&1&1&1&1&1&1&1&1&1\\
 &5&1&1&1&.&.&1&1&1&.&.&1&1&1&1\\
 \\
 &&1a&2a&3a&4a&4b&5a&6a&10a&12a&12b&15a&15b&30a&30b\\
\\
X.11&&2&-2&-1&2i&-2i&2&1&-2&-i&i&-1&-1&1&1\\
X.12&&2&-2&-1&-2i&2i&2&1&-2&i&-i&-1&-1&1&1\\
X.15&&4&-4&-2&.&.&-1&2&1&.&.&\alpha&\ov\alpha&-\alpha&-\ov\alpha\\
X.16&&4&-4&-2&.&.&-1&2&1&.&.&\ov\alpha&\alpha&-\ov\alpha&-\alpha\\
\end{array}
$$

We note that $B$ has $2$ irreducible modules and $2$ weights, in conformity with Alperin's weight conjecture \cite{Al}. However the irreducible modules are self-dual and the weights are duals of each other. This shows that there is no obvious `real' version of the weight conjecture for $p$-blocks, when $p\ne2$.

Consider the inclusion of groups $N_0<\op{PSL}_2(11)<M_{11}$, where $N_0\cong{11:5}$. The principal $11$-blocks each have $5$ non-exceptional characters. It is somewhat surprising that the number of real non-exceptional characters in these blocks is $1$, $5$ and $3$, respectively.


\medskip
{\bf Example 5:} Finally $B$ may have a real exceptional character but no real non-exceptional characters. For example let $B$ be the $5$-block containing the four faithful irreducible characters of SmallGroup$(20,1)=\langle a,b\mid a^5,b^4,a^b=a^{-1}\rangle$. The two exceptional characters have F-S indicators $-1$, but neither of the two non-exceptional characters is real. Here is the character table of $B$, with $\beta=(-1+\sqrt{5})/2$ and $*\beta=(-1-\sqrt{5})/2$:
$$
\begin{array}{lrrrrrrrrr}
 &2&2&2&2&2&1&1&1&1\\
 &5&1&1&.&.&1&1&1&1\\
 \\
 &&1a&2a&4a&4b&5a&5b&10a&10b\\
 \\
X.3&&1&-1&i&-i&1&1&-1&-1\\
X.4&&1&-1&-i&i&1&1&-1&-1\\
X.5&&2&-2&.&.&\beta&{*\beta}&-\beta&-{*\beta}\\
X.6&&2&-2&.&.&{*\beta}&\beta&-{*\beta}&-\beta\\
\end{array}
$$

\section{Miscellaneous results}\label{S:misc}

We need general results from representation theory, some of which are not so well-known. So in this section $p$ is a prime and $B$ is a $p$-block of a finite group $G$.

Let $\chi$ be an irreducible character in $B$, let $x$ be a $p$-element of $G$ and let $y$ be a $p$-regular element of $\op{C}_G(x)$. Then
$$
\chi(xy)=\sum_{\varphi}d_{\chi,\varphi}^{(x)}\varphi(y),
$$
where $\varphi$ ranges over the irreducible Brauer characters in blocks of $\op{C}_G(x)$ which Brauer induce to $B$, and each $d_{\chi,\varphi}^{(x)}$ is an algebraic integer, called a generalized decomposition number; if $x=1$, $\varphi$ is an irreducible Brauer character in $B$ and $d_{\chi,\varphi}^{(x)}$ is simplified to $d_{\chi,\varphi}$. It is an integer called an ordinary decomposition number of $B$.

Brauer \cite[Theorem (4A)]{B} used his Second Main Theorem to prove the following remarkable `local-to-global' formula for F-S indicators:
\begin{equation}\label{E:BrauerIII}
\sum_{\chi}\epsilon(\chi)d_{\chi,\varphi}^{(x)}=
\sum_{\psi}\epsilon(\psi)d_{\psi,\varphi}^{(x)},
\end{equation}
where $\chi$ ranges over the irreducible characters in $B$ and $\psi$ ranges over the irreducible characters in blocks of $\op{C}_G(x)$ which Brauer induce to $B$. We have previously used this formula to determine the F-S indicators of the irreducible characters in 2-blocks with a cyclic, Klein-four or dihedral defect group.

Our next result relies on Clifford theory. However it was inspired by (and can be proved using) the notion of a weakly real $2$-block, as introduced in \cite{M1}. Suppose that $N$ is a normal subgroup of $G$ and $\phi\in\op{Irr}(N)$, with stabilizer $G_\phi$ in $G$.  If $G_\phi\subseteq H\subseteq G$, the Clifford correspondence is a bijection $\op{Irr}(G\mid\phi)\leftrightarrow\op{Irr}(H\mid\phi)$ such that $\chi\leftrightarrow\psi$ if and only if $\langle\chi{\downarrow_H},\phi\rangle\ne0$ or $\chi=\psi{\uparrow^G}$. The stabilizer of $\{\phi,\ov\phi\}$ in $G$ is called the extended stabilizer of $\phi$, here denoted by $G_\phi^*$. So $|G_\phi^*:G_\phi|\leq2$, with equality if and only if $\phi\ne\ov\phi$ but $\phi$ is $G$-conjugate to $\ov\phi$. If $G_\phi^*\subseteq H$ it is easy to see that $\chi$ is real if and only if $\psi$ is real. Moreover in this case $\epsilon(\chi)=\epsilon(\psi)$.

We need one other idea. Suppose that $T$ is a degree $2$ extension of $G$. Then the {\em Gow indicator} \cite[2.1]{G} of a character $\chi$ of $G$ with respect to $T$ is defined to be
$$
\epsilon_{T/G}(\chi):=\frac{1}{|G|}\sum_{t\in T\backslash G}\chi(t^2).
$$
Clearly $\epsilon(\chi{\uparrow^T})=\epsilon(\chi)+\epsilon_{T/G}(\chi)$. Just like the F-S indicator, $\epsilon_{T/G}(\chi)=0,\pm1$, for each $\chi\in\op{Irr}(G)$. Moreover $\epsilon_{T/G}(\chi)\ne0$ if and only if $\chi$ is $T$-conjugate to $\ov\chi$.

\begin{Lemma}\label{L:odd_normal}
Let $N$ be a normal odd order subgroup of\/ $G$ and let $\phi\in\op{Irr}(N)$. Suppose that $G_\phi^*$ does not split over $G_\phi$. Then there exists $\chi\in\op{Irr}(G\mid\phi)$ such that\/ $\epsilon(\chi)=-1$.
\begin{proof}
We first show that there exists $\psi\in\op{Irr}(G\mid\phi)$ such that $\epsilon(\psi)=+1$. For let $S$ be a Sylow $2$-subgroup of $G$. As $\phi{\uparrow^G}$ vanishes on the $2$-singular elements of $G$, we have $(\phi{\uparrow^G}){\downarrow_S}=\frac{\phi(1)|G|}{|N||S|}\rho_S$, where $\rho_S$ is the regular character of $S$. Now $\frac{\phi(1)|G|}{|N||S|}$ is an odd integer. So $\langle(\phi{\uparrow^G}){\downarrow_S},1_S\rangle$ is odd. Moreover $\phi{\uparrow^G}$ is a real character of $G$. So $\langle(\phi{\uparrow^G}),\psi\rangle=\langle(\phi{\uparrow^G}),\ov\chi\rangle$, for each $\psi\in\op{Irr}(G)$. Pairing each irreducible character of $G$ with its complex conjugate, we see that there exists a real-valued $\psi\in\op{Irr}(G\mid\phi)$ such that $\langle\psi{\downarrow_S},1_S\rangle$ is odd. Then $\epsilon(\psi)=\epsilon(1_S)=+1$.

Following the discussion before the lemma, we may assume that $G=G_\phi^*$.  So $|G:G_\phi|=2$. Next suppose that $g\in G$ and $\phi{\uparrow^{G_\phi}}(g^2)\ne0$. Write $g=xy=yx$, where $x$ is a $2$-element and $y$ is a $2$-regular element. Then $g^2=x^2y^2$. As $\phi{\uparrow^{G_\phi}}$ vanishes off $N$, we have $x^2=1$ and $y^2\in N$. So $x\in G_\phi$, as $G_\phi$ contains all involutions in $G$. Moreover $y\in N$, as $y$ has odd order. Thus $g\in G_\phi$, whence
$$
\epsilon_{G/G_\phi}(\phi{\uparrow^{G_\phi}})=\frac{1}{|G_\phi|}\sum_{g\in G\backslash G_\phi}\phi{\uparrow^{G_\phi}}(g^2)=0.
$$
Now $\op{Irr}({G_\phi\mid\phi})$ contains no real characters, as $\phi\ne\ov\phi$. So $\epsilon(\phi{\uparrow^G})=\epsilon_{G/G_\phi}(\phi{\uparrow^{G_\phi}})+\epsilon(\phi{\uparrow^{G_\phi}})=0$. Equivalently
$$
\sum_{\chi\in\op{Irr}(G)}\langle\phi{\uparrow^G},\chi\rangle\epsilon(\chi)=0.
$$
Together with the fact that $\langle\phi{\uparrow^G},\psi\rangle\epsilon(\psi)>0$, this implies that $\langle\phi{\uparrow^G},\chi\rangle\epsilon(\chi)<0$, for some $\chi\in\op{Irr}(G)$. Thus $\chi\in\op{Irr}(G\mid\phi)$ and $\epsilon(\chi)=-1$, which completes the proof.
\end{proof}
\end{Lemma}

It is well-known that each $G$-invariant irreducible character of a normal subgroup of $G$ extends to $G$, when the quotient group is cyclic.
 
\begin{Lemma}\label{L:cyclic_real_extension}
Suppose that $N$ is a normal subgroup of $G$ such that $G/N$ is cyclic and of even order. Let $\varphi\in\op{Irr}(N)$ be real and $G$-invariant. Then $\varphi$ has a real extension to $G$ if and only if\/ $\varphi$ has a real extension to $T$, where $N\subset T\subseteq G$ and $T/N$ has order $2$. 
\end{Lemma}
\begin{proof}
The `only if' part is obvious. So assume that $\varphi$ has a real extension to $T$. Then both extensions of $\varphi$ to $T$ are real. Let $\omega$ be a generator of the abelian group $\op{Irr}(G/N)$ and let $\chi$ be any extension of $\varphi$ to $G$. Then $\omega^i\chi$, $i\geq0$ give all extensions of $\varphi$ to $G$. Here $\omega^i=\omega^j$ if and only if $i\equiv j\Mod{|G/N|}$.
 
As $\ov\chi$ lies over $\varphi$, we have $\ov\chi=\omega^i\chi$, for some $i\geq0$. Now $\chi{\downarrow_T}$ is an extension of $\varphi$ to $T$ and $\ov\chi{\downarrow_T}=(\omega^i{\downarrow_T})(\chi{\downarrow_T})$. As $\chi{\downarrow_T}$ is real, it follows that $\omega^i{\downarrow_T}$ is trivial. So $T\subseteq\op{ker}(\omega^i)$, whence $i\equiv 2j\Mod{|G/N|}$, for some $j\geq0$. Now $\ov{\omega^j\chi}=\omega^{i-j}\chi=\omega^j\chi$. So $\omega^j\chi$ is a real extension of $\varphi$ to $G$.
\end{proof}

Notice that in this context $\varphi$ has a real extension to $T$ if and only if $\epsilon(\varphi)=\epsilon_{T/N}(\varphi)$. When $G/N$ has even order, but is not cyclic, and $\varphi$ is a real irreducible character of $N$ which extends to $G$, it is not clear whether there is a sensible sufficient criteria for $\varphi$ to have a real extension to $G$.

Finally we need the following consequence of the first orthogonality relation:

\begin{Lemma}\label{L:gen_sum}
Let $W\subseteq X\subseteq Y$ be finite abelian groups. Then for $\lambda\in\op{Irr}(Y)$ we have
$$
\sum_{x\in X\backslash W}\lambda(x)=
\left\{\begin{array}{rl}
        |X|-|W|,&\quad\mbox{if\/ $X\subseteq\op{ker}(\lambda)$.}\\
               -|W|,&\quad\mbox{if\/ $W\subseteq\op{ker}(\lambda)$ but $X\not\subseteq\op{ker}(\lambda)$.}\\
               0,&\quad\mbox{if\/ $W\not\subseteq\op{ker}(\lambda)$.}\\
       \end{array}
\right.
$$
\end{Lemma}

\section{The Brauer tree and its real-stem}\label{S:real-stem}

From now on $G$ is a finite group, $p$ is an odd prime and $B$ is a real $p$-block of $G$ which has a cyclic defect group. To avoid trivialities we assume that the defect group is non-trivial.

Dade asserts \cite[Theorem 1, Part 2]{D} that each decomposition number in $B$ is either $0$ or $1$. The Brauer tree of $B$ is a planar graph with edges labelled by the irreducible Brauer character in $B$ and with vertices labelled by the irreducible characters in $B$ (the exceptional characters in $B$ label a single `exceptional' vertex). The edge labelled by an irreducible Brauer character $\theta$ meets the vertex labelled by an irreducible character $\chi$ if and only if the decomposition number $d_{\chi,\theta}$ is not $0$.

When $B$ is real, complex conjugation acts on the Brauer tree of $B$, and in particular fixes the exceptional vertex. However, as we have seen in Examples 2,3 and 4 above, $B$ may have no real exceptional characters. So we restate \cite[VII,9.2]{F} in the following more precise fashion:

\begin{Lemma}\label{L:real_stem}
The subgraph of the Brauer tree of $B$ consisting of the exceptional vertex and those vertices and edges which correspond to real characters and Brauer characters is a straight line segment.
\end{Lemma}

Feit calls this line segment the real-stem of $B$. An easy consequence is:

\begin{Corollary}\label{C:number_real_non_exceptional}
The number of real non-exceptional characters in $B$ equals the number of real irreducible Brauer characters in $B$.
\begin{proof}
Suppose that $B$ has $r$ real irreducible Brauer characters. Then the real-stem of the Brauer tree has $r$ edges and $r+1$ vertices. One of these is the exceptional vertex. So $B$ has $r$ real non-exceptional characters.  
\end{proof}
\end{Corollary}

Let $\theta$ be a real irreducible $p$-Brauer character of a finite group $G$. As $p$ is odd, the $G$-representation space of $\theta$ affords a non-degenerate $G$-invariant bilinear form which is either symmetric or skew-symmetric. Given the symmetry groups of such forms, we refer to $\theta$ as being of orthogonal or symplectic type. Thompson and Willems \cite[2.8]{W} proved that there is a real irreducible character $\chi$ of $G$ such that $d_{\chi,\theta}$ is odd. Moreover $\theta$ has orthogonal type if $\epsilon(\chi)=+1$ or symplectic type if $\epsilon(\chi)=-1$. This implies that $\epsilon(\psi)=\epsilon(\chi)$, for all real irreducible characters $\psi$ such that $d_{\psi,\theta}$ is odd.

\begin{proof}[Proof of part (ii) of Theorem \ref{T:main}]
Let $X$ and $Y$ be real non-exceptional characters which lie on the same side of the exceptional vertex in the real-stem of $B$. Then by Lemma \ref{L:real_stem} there is a sequence $X=X_0,X_1,\dots,X_n=Y$ of real non-exceptional characters and a sequence $\theta_1,\dots,\theta_n$ of real irreducible Brauer characters such that $d_{X_{i-1},\theta_i}=1=d_{X_i,\theta_i}$, for $i=1,\dots,n$. The Thompson-Willems result implies that $\epsilon(X_{i-1})=\epsilon(X_i)$, for $i=1,\dots,n$. So $\epsilon(X)=\epsilon(Y)$. This gives part (ii) of Theorem \ref{T:main}.
\end{proof}

A similar argument gives the following weak form of parts (i) and (iii) of Theorem \ref{T:main}:

\begin{Lemma}\label{L:main(iii)}
If $B$ has a real exceptional character and a real non-exceptional character, then all real irreducible characters in $B$ have the same F-S indicators.
\end{Lemma}

Notice that if $B$ is the principal $p$-block of a group with a cyclic Sylow $p$-subgroup, and $B$ has an irreducible character with F-S indicator $-1$ (e.g. the principal $7$-block of $\op{U}(3,3)$) then the lemma implies that $B$ has no real exceptional characters.

\section{The exceptional characters}\label{S:exceptional}

We outline some results from \cite{D} using the language of subpairs. See \cite[Chapter 5.9]{NT} for a full description of the theory. We then prove results about the local blocks in $B$, in Proposition \ref{P:real}, and the exceptional characters in $B$, in Proposition  \ref{P:exceptional}. This allows us to prove parts (i), (iii) and (iv) of Theorem \ref{T:main}.

Recall that $B$ is a $p$-block with a non-trivial cyclic defect group $D$. Write $|D|=p^a$, where $a>0$, and let $1\subset D_{a-1}\subset D_{a-2}\subset\dots\subset D_1\subset D_0=D$ be the complete list of subgroups of $D$. So $[D:D_i]=p^i$, for $i=0,\dots,a-1$. Set $C_i=C_G(D_i)$ and $N_i=N_G(D_i)$. So $C_0\subseteq C_1\subseteq\ldots\subseteq C_{a-1}$, and $N_0\subseteq N_1\subseteq\ldots\subseteq N_{a-1}$.

As $p$ is odd, $\op{Aut}(D_i)$ is a cyclic group of order $p^{a-i-1}(p-1)$. So $N_i/C_i$ is a cyclic group whose order divides $p^{a-i-1}(p-1)$. Moroever the centralizer of $D_i$ in $\op{Aut}(D)$ has order $p^i$. So $C_i\cap N_0/C_0$ is a cyclic $p$-group. We note that the unique involution in $\op{Aut}(D)$ inverts every element of $D$.

Fix a Sylow $B$-subpair $(D,b_0)$. So $b_0$ is a $p$-block of $C_0$ such that $b_0^G=B$ and the pair $(D,b_0)$ is uniquely determined up to $G$-conjugacy. Set $b_i:=b_0^{C_i}$, for $i=1,\dots,a-1$. Then by \cite[5.9.3]{NT} the lattice of $B$-subpairs contained in $(D,b_0)$ is
\begin{equation}\label{E:all_subpairs}
(1,B)\subset(D_{a-1},b_{a-1})\subset\dots\subset(D_1,b_1)\subset(D,b_0).
\end{equation}

Set ${E:=\op{N}(D,b_0)}$, the stabilizer of $b_0$ in $N_0$. Then ${e:=|E:C_0|}$ is called the inertial index of $B$. Now ${p\not\hspace{.1cm}\mid e}$, by Brauer's extended first main theorem. So $e\mid(p-1)$. Let $x\in E$. Then $D_i^x=D_i$. As $(D_i,b_i),(D_i,b_i^x)\subseteq(D,b_0)$, it follows from \eqref{E:all_subpairs} that $b_i^x=b_i$. So $EC_i\subseteq\op{N}(D_i,b_i)$. Conversely let $n\in\op{N}(D_i,b_i)$. As $(D,b_0)$ and $(D,b_0)^n$ are Sylow $b_i$-subpairs (in the group $C_i$), there is $c\in C_i$ such that $nc_i\in E$. This shows that $\op{N}(D_i,b_i)\subseteq EC_i$. This recovers Dade's observation that $\op{N}(D_i,b_i)=EC_i$.

Now $E\cap C_i/C_0$ is a subgroup of $C_i\cap N_0/C_0$ and a quotient of the group $E/C_0$. As $C_i\cap N_0/C_0$ is a $p$-group and $E/C_0$ has $p'$-order, we deduce that $E\cap C_i=C_0$. It follows from this $EC_i/C_i\cong E/C_0$, and in particular $|EC_i:C_i|=e$.

By \cite[Theorem 1, Part 1]{D} $B$ has $e$ irreducible Brauer characters, listed as $\chi_1,\dots,\chi_e$. Each $b_i$ has inertial index $1$. So $b_i$ has a unique irreducible Brauer character, denoted $\varphi_i$. 

From the above discussion there are $|N_i:EC_i|=\frac{|N_i:C_i|}{e}$ distinct blocks of\/ $C_i$ which induce to $B$, namely $b_i^\tau$ as $\tau$ ranges over $N_i/EC_i$. Also there are $\frac{p^{a-i}-p^{a-i-1}}{|N_i:C_i|}$ conjugacy classes of\/ $G$ which contain a generator of\/ $D_i$. So $B$ has $\frac{p^{a-i}-p^{a-i-1}}{e}$ subsections $(x,b)$, with $D_i=\langle x\rangle$. A consequence of Brauer's second main theorem \cite[5.4.13(ii)]{NT} is that the number of irreducible characters in a block equals the number of columns in the block.

\begin{Lemma}\label{L:columns}
A complete set of columns of\/ $B$ is
$$
(1,\chi_1),\dots,(1,\chi_e),\qquad(x_i^{\sigma_i},\varphi_i^{n_i}),\quad i=0,\dots,a-1.
$$
Here $x_i$ is a fixed generator of $D_i$, $\sigma_i$ ranges over a set of representatives for the cosets of the image of\/ $N_i/C_i$ in $\op{Aut}(D_{a-i})$ and $n_i$ ranges over a set of representatives for the cosets of $EC_i$ in $N_i$. In particular $k(B)=e+\frac{p^a-1}{e}$.
\end{Lemma}

Let $\Lambda$ be a set of representatives for the $\frac{p^a-1}{e}$ orbits of $E$ on $\op{Irr}(D)^\times$. Then
\begin{equation}\label{E:exceptional}
\op{Irr}(B)=\{X_1,\dots,X_e\}\,\bigcup\,\{X_\lambda\mid\lambda\in\Lambda\}.
\end{equation}
Also set $X_0:=\sum_{\lambda\in\Lambda}X_\lambda$. Dade refers to the $X_\lambda$ as the exceptional characters of $B$.

Notice that as $\ell(b_i)=1$, $b_i$ is real if and only if $\varphi_i$ is real. The next two propositions are relatively elementary.

\begin{Proposition}\label{P:real}
All the blocks $b_0,b_1,\dots,b_{a-1}$ are real or none of them are real.
\begin{proof}
We have $(b_i^o)^G=B^o=B$. So $(D,b_0)$ and $(D,b_o^o)$ are Sylow $B$-subpairs, and there is $n\in\op N_0$ such that $b_0^o=b_0^n$.

Suppose that $b_j$ is real, for some $j=0,\dots,a-1$. As $(D_j,b_j^n),(D_j,b_j^o)\subset (D_0,b_0^o)$, it follows from \eqref{E:all_subpairs} that $b_j^n=b_j^o=b_j$. So $n\in\op{N}(D_j,b_j)=EC_j$. Write $n=ec$, where $e\in E$ and $c\in C_j$. Then $c=e^{-1}n\in C_j\cap N_0$ and $b_0^c=b_0^n=b_0^o$. So $c^2\in C_j\cap E=C_0$. But $C_j\cap N_0/C_0$ has odd order, as it is a $p$-group. So $c\in C_0$, which shows that $n\in E$. As $b_0^n=b_0$, it follows that $b_0$ is real.

Now let $i=0,\dots,a-1$. Then $(D_i,b_i),(D_i,b_i^o)\subset(D_0,b_0)=(D_0,b_0^o)$. So $b_i=b_i^o$, for $i=0,\dots,a-1$, using \eqref{E:all_subpairs}. This shows that all $b_0,\dots,b_{a-1}$ are real.
\end{proof}
\end{Proposition}

We showed in \cite[1.1]{M3} that the number of real irreducible characters in a block equals the number of real columns in the block. Here $(x,\varphi)$ is real if $x^g=x^{-1}$ and $\varphi^g=\ov\varphi$, for some $g\in G$. 

Let $i=0,\dots,a-1$. As $b_i$ has inertial index $1$, it has $|D|$ irreducible characters. Modifying \cite[p26]{D} we use the notation
\begin{equation}\label{E:bi}
\op{Irr}(b_i)=\{X_{i,\lambda}'\mid\lambda\in\op{Irr}(D)\}.
\end{equation}
Here $X_{i,1}'$ is the unique non-exceptional character in $b_i$, and all characters $X_{i,\lambda}'$ with $\lambda\ne1$ are exceptional. Suppose that $b_i$ is real. The columns of $b_i$ are $(d,\varphi_i)$, for $d\in D$. As $C_i$ acts trivially on the columns, the only real column is $(1,\varphi_i)$. So $X_{i,1}'$ is the only real irreducible character in $b_i$.

We will refine the next result in part (i) of Theorem \ref{T:main}:

\begin{Proposition}\label{P:exceptional}
All exceptional characters in $B$ are real or none are real.
\begin{proof}
It follows from Corollary \ref{C:number_real_non_exceptional} and Lemma \ref{L:columns} that the number of real exceptional characters in $B$ equals the number of real columns $(x,\varphi)$ with $x\in D^\times$ and $\varphi\in\op{IBr}(\op{C}_G(x))$.

Suppose that $B$ has a real exceptional character, and let $(x,\varphi)$ be a real column of $B$, with $x\in D^\times$. Then $\langle x\rangle=D_i$, for some $i=0,\dots,a-1$. As $N_i/C_i$ is abelian, the columns $(x',\varphi_i^{n_i})$ are real, for all generators $x'$ of $D_i$ and all $n_i\in N_i$. In particular $(x_i,\varphi_i)$ is a real column. Choose $n\in N_i$ such that $x_i^n=x_i^{-1}$ and $\varphi_i^n=\ov\varphi_i$. We may suppose that $n^2\in C_i$.

Suppose first that $b_i$ is real. As $\varphi_i=\ov\varphi_i$, $n$ fixes $\varphi_i$ and inverts $D_i$. So $nC_i$ is an involution in $EC_i/C_i$. As $EC_i/C_i\cong E/C_0$, we may assume without loss that $nC_0$ is an involution in $E/C_0$. Now all the blocks $b_0,\dots,b_{a-1}$ are real. Hence all $\varphi_0,\dots,\varphi_{a-1}$ are real. As $n$ inverts $D_j$ and fixes $\varphi_j$, all columns $(x_j,\varphi_j)$ are real. Thus all columns $(x,\varphi)$, with $x\in D^\times$, are real. So all exceptional characters in $B$ are real in this case.

Conversely, suppose that $b_i$ is not real. As $nC_i$ is the unique involution in $N_i/C_i$, but $n\not\in EC_i$, it follows that $|EC_i:C_i|=e$  is odd. Now $(D,b_0)$ and $(D,b_0^o)$ are Sylow $B$-subpairs, but $b_0\ne b_0^o$. So there is $m\in N_0\backslash E$ such that $b_0^m=b_0^o$. As $m^2\in E$ and $|E:C_0|$ is odd, we may choose $m$ so that $m^2\in C_0$. Then $mC_0$ is the unique involution in $N_0/C_0$. In particular $m$ inverts every element of $D$. Let $j=0,\dots,a-1$. Then $(D_j,b_j^m)$ and $(D_j,b_j^o)$ are $B$-subpairs contained in $(D,b_0^o)$. So $b_j^m=b_j^o$ and thus $(d_j,\varphi_j)^m=(d_j^{-1},\ov\varphi_j)$. It follows that all exceptional characters in $B$ are real in this case also.
\end{proof}
\end{Proposition}

Examination of the proof shows that:

\begin{Corollary}\label{C:real_exceptional}
All exceptional characters in $B$ are real if and only if $b_0$ is real and $e$ is even, or $b_0$ is not real and $e$ is odd.
\end{Corollary}

We need some additional notation. Set $\Lambda_u:=\{\lambda\in\Lambda\mid\op{ker}(\lambda)=D_u\}$, for $u=1,\dots,a$. So $|\Lambda_u|=\frac{p^u-p^{u-1}}{e}$. Now choose $\lambda\in\Lambda_u$ and set
$$
\epsilon_u:=\epsilon(X_\lambda).
$$
Note that $X_\lambda$ and $X_\mu$ are Galois conjugates, for all $\lambda,\mu\in\Lambda_u$ (this follows from \cite[part 2 of Theorem 1 and Corollary 1.9]{D}). So $\epsilon_u$ does not depend on $\lambda$.

Recall our notation \eqref{E:bi} for the irreducible characters $X_{i,\lambda}'$ in $b_i$. As already noted, $X_{i,1}'$ is the only possible real irreducible character in $b_i$. We set
$$
\nu_i:=\epsilon(X_{i,1}'),\quad\mbox{for $i=0,\dots,a-1$.}
$$

Now let $i=0,\dots,a-1$ and choose $x\in D_i-D_{i+1}$ and $\rho\in N_i$. According to \cite[Theorem 1, Part 3]{D} there are signs $\varepsilon_0',\varepsilon_0,\varepsilon_1\dots,\varepsilon_e$ and $\gamma_i$ such that  
$$
\begin{array}{lcllcl}
d_{X_\lambda,\varphi_i^\rho}^{(x)}&=&\varepsilon_0\gamma_i\sum\limits_{\tau\in EC_i/C_i}\lambda({^{\rho\tau}}x),
&\quad
d_{X_j,\varphi_i^\rho}^{(x)}&=&\varepsilon_j\gamma_i,\quad\mbox{for $j=1,\dots,e$}\\
d_{X_{i,\lambda}',\varphi_i^\rho}^{(x)}&=&\varepsilon_0'\gamma_i\lambda({^{\rho}}x),
&\quad
d_{X_{i,1}',\varphi_i^\rho}^{(x)}&=&1
\end{array}
$$
Here $EC_i/C_i$ is a set of representatives for the cosets of $C_i$ in $EC_i$. Note that Feit uses the notation $\delta_0=-\varepsilon_0$ and $\delta_j=\varepsilon_j$, for $j=1,\dots,e$. Now let $i=0,\dots,a-1$ and $x\in D_i-D_{i+1}$. Then it follows from \cite[Corollary 1.9]{D} that $X_j(x)=|N_i:EC_i|\varphi_i(1)\delta_j\gamma_i$. So $\delta_j\gamma_i$ is the sign of the integer $X_j(x)$.

There is a nice relationship between the signs $\varepsilon_0,\varepsilon_1\dots,\varepsilon_e$ and the Brauer tree of $B$. Suppose that $j$ and $k$ are adjacent vertices in the Brauer tree. Then $X_j+X_k$ is a principal indecomposable character of $G$. So it vanishes on $D^\times$, and hence $\delta_j+\delta_k=0$ (see \cite[V11, Section 9]{F}). So suppose that there are $d_j$ edges between the vertex $j$ and the exceptional vertex $0$ in the Brauer tree. Then $\delta_j=(-1)^{d_j}\delta_0$. So $\varepsilon_j=(-1)^{d_j-1}\varepsilon_0$, for $j=1,\dots,e$.

We now prove part (i) of our main theorem. But note that this proof does not depend on Propositions \ref{P:real} and \ref{P:exceptional}:

\begin{proof}[Proof of part (i) of Theorem \ref{T:main}]
Applying \eqref{E:BrauerIII}, with $\rho\in N_i$ and $x\in D_i-D_{i+1}$, we get
$$
\sum_{j=1}^e\epsilon(X_j)\varepsilon_j\gamma_i+\sum_{\lambda\in\Lambda}\epsilon(X_\lambda)\varepsilon_0\gamma_i\sum_{\tau\in EC_i/C_i}\lambda({^{\rho\tau}}x)=\nu_i.
$$
Now set $\sigma:=\varepsilon_0\sum_{j=1}^e\epsilon(X_j)\varepsilon_j$. So $\sigma$ is independent of $i$, $\rho$ and $x$. Then the above equality transforms to
$$
\sum_{u=1}^{a}\epsilon_u\sum_{\lambda\in\Lambda_u}\sum_{\tau\in EC_i}\lambda({^{\rho\tau}}x) = \varepsilon_0\gamma_i\nu_i - \sigma,
$$
where the right hand side is independent of $\rho$ and $x$. Let $\rho$ range over a set of representatives for the $\frac{|N_i:C_i|}{e}$ cosets of $EC_i$ in $N_i$ and let $x$ range over a set of representatives for the $\frac{p^{a-i}-p^{a-i-1}}{|N_i:C_i|}$ orbits of $N_i$ on the generators of $D_i$. Then ${^{\rho\tau}}x$ will range over all generators of $D_i$. Summing the resulting equalities gives
$$
\sum_{u=1}^{a}\epsilon_u\sum_{\lambda\in\Lambda_u}\sum_{x\in D_i-D_{i+1}}\lambda(x)=
\left(\frac{p^{a-i}-p^{a-i-1}}{e}\right)\left(\varepsilon_0\gamma_i\nu_i-\sigma\right).
$$
We use $|\Lambda_u|=\frac{p^u-p^{u-1}}{e}$ and Lemma \ref{L:gen_sum} to transform this equality to
$$
(p^{a-i}-p^{a-i-1})\sum_{u=1}^{i}\frac{p^u-p^{u-1}}{e}\epsilon_u-p^{a-i-1}\frac{p^{i+1}-p^i}{e}\epsilon_{i+1}=\frac{p^{a-i}-p^{a-i-1}}{e}(\varepsilon_0\gamma_i\nu_i-\sigma).
$$
After cancelling the factor $\frac{p^{a-i-1}(p-1)}{e}$, we get
\begin{equation}\label{E:equality}
\sum_{u=1}^{i}(p^u-p^{u-1})\epsilon_u-p^i\epsilon_{i+1}=\varepsilon_0\gamma_i\nu_i-\sigma.
\end{equation}
Here $\sum_{u=1}^0(p^u-p^{u-1})\epsilon_u$ is taken to be $0$, when $i=0$. We write down the equalities \eqref{E:equality} for $i=0,1,2,\dots$ in turn:
\begin{equation}\label{E:i=0,...,a-1}
\begin{aligned}
& -\epsilon_1&=&\quad\varepsilon_0\gamma_0\nu_0-\sigma\\
(p-1)\epsilon_1& -p\epsilon_2&=&\quad\varepsilon_0\gamma_1\nu_1-\sigma\\
(p-1)\epsilon_1+(p^2-p)\epsilon_2& -p^2\epsilon_3&=&\quad\varepsilon_0\gamma_2\nu_2-\sigma\\
(p-1)\epsilon_1+(p^2-p)\epsilon_2+(p^3-p^2)\epsilon_3&-p^3\epsilon_4&=&\quad\varepsilon_0\gamma_3\nu_3-\sigma\\
\vdots\\
(p-1)\epsilon_1+(p^2-p)\epsilon_2+\dots+(p^{a-1}-p^{a-2})\epsilon_{a-1}&-p^{a-1}\epsilon_a&=&\quad\varepsilon_0\gamma_{a-1}\nu_{a-1}-\sigma\\
\end{aligned}
\end{equation}
Subtract the first equality from the second to get
$$
p(\epsilon_1-\epsilon_2)=\varepsilon_0(\gamma_1\nu_1-\gamma_0\nu_0).
$$
The left hand side equals $-p,0$ or $p$ and the right hand equals $-2,0$ or $2$. As $p$ is odd, the common value is $0$. So $\epsilon_2=\epsilon_1$ and $\gamma_1\nu_1=\gamma_0\nu_0$. Substitute these values back into all equations in \eqref{E:i=0,...,a-1}. Now subtract the first from the third equality to get
$$
p^2(\epsilon_1-\epsilon_3)=\varepsilon_0(\gamma_2\nu_2-\gamma_0\nu_0).
$$
Once again both sides are $0$. So $\gamma_2\nu_2=\gamma_0\nu_0$ and $\epsilon_3=\epsilon_1$. Proceeding in this way, we get
$$
\epsilon_1=\epsilon_2=\dots=\epsilon_a,\quad \gamma_0\nu_0=\gamma_1\nu_1=\dots=\gamma_{a-1}\nu_{a-1}.
$$
\end{proof}

Following the above proof, and the discussion before the proof, we obtain:

\begin{Corollary}
Suppose that $b_0$ is real and let $D=\langle x\rangle$. Then for each $i=0,\dots,a-1$ and $j=0,\dots,e$, the integer $X_j(x^{p^i})X_j(x)$ has sign $\epsilon(X_{i,1}')\epsilon(X_{0,1}')$.
\end{Corollary}

\medskip
There is no apparent relationship between the F-S indicators $\nu_0,\dots,\nu_{a-1}$:

{\bf Example:} The $2$-nilpotent group $G=\langle a, b, c\mid a^4, a^2=b^2, a^b=a^{-1}, c^9, a^c=b, b^c=ab\rangle$ has isomorphism type $3.\op{SL}(2,3)$. Set $D=\langle c\rangle$. Then $D$ is cyclic of order $9$, with $C_0=D\times\langle a^2\rangle$ and $C_1=G$. Let $\theta$ be the non-trivial irreducible character of $C_0/D$, and let $b_0$ be the $3$-block of $C_0$ which contains $\theta$. Then $\theta=X_{0,1}'$ is the unique non-exceptional character in $b_0$. So $\nu_0=\epsilon(X_{0,1}')=+1$. Set $b_1=b_0^G$. Then $b_1$ also has a unique non-exceptional character $X_{1,1}'$. But now $\nu_1=\epsilon(X_{1,1}')=-1$, as $X_{1,1}'$ restricts to the non-linear irreducible character of $\langle a,b\rangle\cong Q_8$.

This example arises from the fact that the Glauberman correspondence \cite[5.12]{NT} does not preserve the F-S indicators of characters.

\begin{proof}[proof of part (iii) of Theorem \ref{T:main}]
This is an immediate consequence of Lemma \ref{L:main(iii)} and part (i) of Theorem \ref{T:main}. 
\end{proof}

Consider the real-stem of $B$ as a horizontal line segment with $s$ vertices and $s-1$ edges, where $s\geq1$. We label the vertices using an interval $[-\ell,\dots,-2,-1,0,1,2,\dots,r]$ so that $0$ labels the exceptional vertex. Thus $s=r+\ell+1$, and there are $\ell$ real non-exceptional characters on the left of the exceptional vertex, and $r$ on the right (the choice of left and right is unimportant).

As above, $X_0$ is the sum of the exceptional characters in $B$. Now we relabel the non-exceptional characters in $B$ so that $X_i$ is the real non-exceptional character corresponding to vertex $i$, for $i=-\ell,\dots,r$ and $i\ne0$. In view of parts (i) and (ii) of Theorem \ref{T:main} there are signs $\epsilon_\pm$ such that
$$
\epsilon(X_i)=\left\{ 
\begin{array}{ll}
 \epsilon_-,&\quad\mbox{for $i=-\ell,\dots,-1$.}\\
 \epsilon_0,&\quad\mbox{for $i=0$.}\\
 \epsilon_+,&\quad\mbox{for $i=1,\dots,r$.}
\end{array}\right.
$$

Next let $\sigma$ be a generator of $D$. It follows from \cite[Corollary 1.9]{D} that $X_0(\sigma)=-\varepsilon_0\gamma_0|N_0:E|\varphi_0(1)$. So $X_i(\sigma)=(-1)^iX_0(\sigma)$, as $X_i+X_{i+1}$ is a projective character of $G$, for $i=-\ell,\dots,r-1$ (see \cite[VII,2.19(ii)]{F}).

Recall from Section \eqref{S:exceptional} that there are $|N_0:E|$ blocks of $C_0$ which induce to $B$; these are the blocks $b_0^\tau$, where $\tau$ ranges over $N_0/E$. We note also that $X_{0,1}'({^\tau}\sigma)=\varphi_0(1)$. Now \cite[Theorem(4B)]{B} is an immediate consequence of \cite[Theorem(4A)]{B}. In our context, this states that
$$
\sum_{i=-\ell}^r\epsilon(X_i)X_i(\sigma)=|N_0:E|\epsilon(X_{0,1}')X_{0,1}'(\sigma).
$$
In view of the previous paragraph this simplifies to
\begin{equation}\label{E:real-stem}
\sum_{i=1}^{\ell}(-1)^i\epsilon_-+\epsilon_0+\sum_{i=1}^r(-1)^i\epsilon_+=-\varepsilon_0\gamma_0\nu_0.
\end{equation}

We consider a number of cases.

Suppose first that $\epsilon_0\ne0$. Then $\epsilon_-=\epsilon_0=\epsilon_+$, by part (iii) of Theorem \ref{T:main}. So \eqref{E:real-stem} becomes
\begin{equation}\label{E:first}
-\varepsilon_0\gamma_0\nu_0\epsilon_0=
\left\{
\begin{array}{cl}
 (-1)^\ell,&\quad\mbox{if $s$ is odd.}\\
 0,&\quad\mbox{if $s$ is even.}
\end{array}
\right.
\end{equation}
In particular $b_0$ is not real if $s$ is even. As $e$ is odd when $s$ is even, this already follows from Corollary \ref{C:real_exceptional}.

Suppose then that $\epsilon_0=0$. Now \eqref{E:real-stem} evaluates as
\begin{equation}\label{E:second}
-\varepsilon_0\gamma_0\nu_0=
\left\{
\begin{array}{cl}
 \epsilon_-,&\quad\mbox{if $\ell$ is odd and $r$ is even.}\\
 \epsilon_-+\epsilon_+,&\quad\mbox{if $\ell$ and $r$ are both odd.}\\
 \epsilon_+,&\quad\mbox{if $\ell$ is even and $r$ is odd.}\\
 0,&\quad\mbox{if $\ell$ and $r$ are both even.}\\
\end{array}
\right.
\end{equation}

\begin{proof}[proof of part (iv) of Theorem \ref{T:main}]
The hypothesis is that $\epsilon_0=0$, at least one of $\epsilon_-,\epsilon_+$ is not zero and $\ell\equiv r\equiv1\Mod2$. Now $B$ has $e$ non-exceptional characters, of which $\ell+r$ are real-valued. So $e\equiv\ell+r$ is even. Then $b_0$ is not real, according to Corollary \ref{C:real_exceptional}. This in turn implies that $\nu_0=0$. So $\epsilon_-+\epsilon+=0$, according to \eqref{E:second}. We conclude that $\epsilon_-\epsilon_+=-1$, which gives the conclusion of (iv).
\end{proof}

\section{Passing from $B$ to its canonical character}

Let $i=0,\dots,a-1$. Then $N_i$ contains the normalizer $N_0$ of $D$ in $G$. So by Brauer's first main theorem there is a unique $p$-block $B_i$ of $N_i$ such that $B_i^G=B$. As $(B_i^o)^G=B^o=B$, the uniqueness forces $B_i^o=B_i$. Now $B_i$ has defect group $D$ and inertial index $e=|EC_i:C_i|$. So $\ell(B_{a-1})=e$ and $k(B_{a-1})=e+\frac{p^a-1}{e}$. We first consider the block $B_{a-1}$ of the largest subgroup $N_{a-1}$. Following \cite[Section 7]{D}, write
$$
\op{IBr}(B_{a-1})=\{\tilde\chi_1,\dots,\tilde\chi_e\},\quad \op{Irr}(B_{a-1})=\{\tilde X_1,\dots,\tilde X_e\}\,\bigcup\,\{\tilde X_\lambda\mid\lambda\in\Lambda\},
$$
and set $\tilde X_0=\sum\tilde X_\lambda$.

\begin{Proposition}\label{P:Ba-1}
The exceptional characters in $B$ and $B_{a-1}$ have the same F-S indicators.
\end{Proposition}

\begin{proof}
Suppose first that $|\Lambda|\geq2$. According \cite[(7.2)]{D} there is a sign $d$ such that
$$
(\tilde X_{\lambda}-\tilde X_{\mu})^G=
d(X_{\lambda}-X_{\mu}),\quad\mbox{for all $\lambda,\mu\in\Lambda$.}
$$
It follows that $\langle \tilde X_{\lambda},X_\lambda\rangle$ or $\langle \tilde X_{\mu},X_\lambda\rangle$ is odd. So in view of part (i) of Theorem \ref{T:main}, the conclusion holds in this case.

From now on we suppose that $|\Lambda|=1$. Then $E$ has a single orbit on $\op{Irr}(D)^\times$, which forces $|D|=p$ and $e=p-1$. As $\tilde X_0$ is the unique exceptional character in $B_{a-1}$, it is real valued. Then it follows from part (iii) of Theorem \ref{T:main} that all real irreducible characters in $B_{a-1}$ have the same F-S indicators.

Now by \cite[(7.3), (7.8), first two paragraphs of p40]{D}, there is a sign $\varepsilon_0'$ such that
$$
(\tilde X_0-\sum_{i=1}^{p-1}\tilde X_i){^G}=\varepsilon_0'\sum_{i=0}^{p-1}\varepsilon_i X_i.
$$
Here $\varepsilon_0,\dots,\varepsilon_{p-1}$ are as introduced earlier and $X_0$ can be chosen to be real, as $p$ is odd. Taking inner-products of characters, and reading modulo $2$, we see that $\langle\tilde X_i^G,X_0\rangle$ is odd, for some real $\tilde X_i$. So $\epsilon(\tilde X_i)=\epsilon(X_0)$. Then by the previous paragraph $\epsilon(\tilde X_0)=\epsilon(X_0)$.
\end{proof}

\begin{Proposition}\label{P:B0}
All exceptional characters in $B_0,\dots, B_{a-1}$ and $B$ have the same F-S indicators.
\end{Proposition}

\begin{proof}
We prove this by induction on $|D|$. The base case $|D|=p$ holds, by Proposition \ref{P:Ba-1}. Suppose that $|D|>p$. We assume that the conclusion holds for all $p$-blocks with a cyclic defect group of order strictly less than $|D|$.

We use the bar notation for subgroups and objects associated with the quotient group $N_{a-1}/D_{a-1}$. Let $i=0,\dots,a-1$. Then $\ov N_i$ is the normalizer of $\ov D_i$ in $\ov N_{a-1}$. As $C_i$ centralizes $D_{a-1}$, Theorem 5.8.11 of \cite{NT} shows that $b_i$ dominates a unique block $\ov{b_i}$ of $\ov C_i$. Moreover $\ov b_i$ has cyclic defect group $\ov D$. Now $b_i$ has the unique irreducible Brauer character $\varphi_i$, and we can and do identify $\varphi_i$ with the unique irreducible Brauer character in $\ov b_i$. Then the inertia group of $\ov b_i$ in $\ov N_i$ is the inertia group of $\varphi_i$ in $\ov N_i$, which is $\ov{EC_i}$.

According to \cite[Section 4]{D}, there is a unique $p$-block  of $N_i$, denoted here by $\ov{B_i}$, which lies over $\ov{b_i}$. Moreover $\ov{B_i}$ has cyclic defect group $\ov{D}$. As inflation and induction of characters commute, this block is dominated by $B_i$. Now $B_i$ and $\ov{B_i}$ have the same inertial index as $|EC_i:C_i|=|\ov{EC_i}:\ov{C_i}|$. So by inflation $\op{IBr}(\ov{B_i})=\op{IBr}(B_i)$. In particular $\ov{B_i}$ is the unique block of $\ov{N_i}$ that is dominated by $B_i$. Also by inflation $\op{Irr}(\ov{B_i})\subseteq\op{Irr}(B_i)$.

As $|\ov D|<|D|$, all exceptional characters in $\ov B_0,\dots,\ov B_{a-1}$ have the same F-S indicators, by our inductive hypothesis. But the inclusion $\op{Irr}(\ov{B_i})\subseteq\op{Irr}(B_i)$ identifies the exceptional characters in $\ov B_i$ with exceptional characters in $B_i$. It now follows from part (i) of Theorem \ref{T:main} that all exceptional characters in $B_0,\dots,B_{a-1}$ have the same F-S indicators.
\end{proof}

Recall that $b_0$ has a unique irreducible Brauer character $\varphi_0$. This is the canonical character of $B$, in the sense of \cite[5.8.3]{NT}. For the next theorem, we simplify the notation of \eqref{E:bi} for the irreducible characters in $b_0$ by writing $\chi_\lambda$ in place of $X_{0,\lambda}'$, for all $\lambda\in\op{Irr}(D)$. Then according to W. Reynolds \cite[5.8.14]{NT}, for $c\in C_0$ we have
\begin{equation}\label{E:chi_lambda}
\chi_\lambda(c)=\left\{\begin{array}{ll}
                        \lambda(c_p)\varphi_0(c_{p}'),&\quad\mbox{if $c_p\in D$.}\\
                        0,&\quad\mbox{if $c_p\not\in D$.}
                       \end{array}
\right.
\end{equation}
Then $\op{Irr}(b_0)=\{\chi_\lambda\mid\lambda\in\op{Irr}(D)\}$. Notice that $\chi_1$ is the unique irreducible character in $b_0$ whose kernel contains $D$.

\begin{Theorem}\label{T:canonical}
Suppose that $B$ has a real exceptional character. Then $N_0/C_0$ has a unique subgroup $T/C_0$ of order $2$, and all exceptional characters in $B$ have F-S indicator equal to the Gow indicator $\epsilon_{T/C_0}(\chi_1)$.
\begin{proof}
Recall that $B$ has a real exceptional character if $b_0$ is real and $e$ is even, or if $b_0$ is not real and $e$ is odd. In both these cases $|N_0:C_0|$ is even. As $N_0/C_0$ is also cyclic, it has a unique subgroup $T/C_0$ of order $2$.

In view of Proposition \ref{P:B0}, we may assume that $G=N_0$. So $B=B_0$, $D$ and $C_0$ are normal subgroups of $G$ and $E$ is the stabilizer of $b_0$ in $G$. Then $\Lambda$ is a set of representatives for the orbits of $N_0$ on $\op{Irr}(D)^\times$. Set $E^*$ as the stabilizer of $\{b_0,b_0^o\}$ in $G$. Clifford correspondence defines a bijection between the irreducible characters of $E^*$ which lie over $b_0$ and the irreducible characters in $B$. This bijection preserves reality, and hence F-S indicators. So from now on we assume that $G=E^*$.

As $\chi_1$ is invariant in $E$ and $E/C_0$ is cyclic, $\chi_1$ has $e$ extensions to $E$, which we denote by $\eta_1,\dots,\eta_e$. Then $X_i:=\eta_i^G$, for $i=1,\dots,e$, give the $e$ non-exceptional characters in $B$. Moreover $X_\lambda:=\chi_\lambda^G$, for all $\lambda\in\Lambda$, give the exceptional characters in $B$.

Following Corollary \ref{C:real_exceptional}, there are three cases we must consider:

\medskip
{\bf Case 1:} $b_0$ is real, $e$ is even and $B$ has real non-exceptional characters. Then according to part (iii) of Theorem \ref{T:main} all real irreducible characters in $B$ have the same F-S indicators. We choose notation so that $X_1$ is real. As $X_1{\downarrow_T}$ is a real extension of $\chi_1$ to $T$, it follows that $\epsilon(X_1)=\epsilon(X_1{\downarrow_T})=\epsilon_{T/C_0}(\chi_1)$. This concludes Case 1.

\medskip
{\bf Case 2:} $b_0$ is real, $e$ is even but $B$ has no real non-exceptional characters. As $\chi_1$ does not extend to a real character of $E$, it does not extend to a real character of $T$, according to Lemma \ref{L:cyclic_real_extension}. So $\epsilon_{T/C_0}(\chi_1)=-\epsilon(\chi_1)$, by the definition of the Gow indicator.

Now consider the notation used in the proof of part (i) of Theorem \ref{T:main}. Here $C_i=C_0$ and $\varphi_i=\varphi_0$ and $X_{i,1}'=\chi_1$, for $i=0,\dots,a-1$. If $\lambda\in\Lambda$ then $(X_\lambda){\downarrow_{C_0}}=\sum_{\tau\in G/C_0}\chi_{\lambda^\tau}$. So $d_{X_\lambda,\varphi_i}^{(x)}=\sum_{\tau\in G/C_0}\lambda({^{\tau}}x)$, for all $x\in D^\times$. This means that $\varepsilon_0\gamma_i=1$, for $i=0,\dots,a-1$. Now in \eqref{E:i=0,...,a-1}, the term $\sigma$ is $0$, as none of $X_1,\dots,X_e$ are real. So the first equation in \eqref{E:i=0,...,a-1} simplifies here to $-\epsilon(X_\lambda)=\epsilon(\chi_1)$, for all $\lambda\in\Lambda_1$. So $\epsilon(X_\lambda)=\epsilon_{T/C_0}(\chi_1)$, for all $\lambda\in\Lambda$, by the previous paragraph and Proposition \ref{P:B0}.


\medskip
{\bf Case 3:} The final case is that $b_0$ is not real and $e$ is odd. As $B$ has an odd number $e$ of non-exceptional characters, at least one of them must be real valued. So we assume that $X_1$ is real. Then, just as in Case 1, all real irreducible characters in $B$ have the same F-S indicators.

As $|E:C_0|$ is odd and $|G:E|=2$, we have $G/C_0=E/C_0\times T/C_0$. Now $T/C_0$ conjugates $\op{Irr}(b_0)$ into $\op{Irr}(b_0^o)$. So $\chi_1$ is $T$-conjugate to $\ov\chi_1$. In particular $\chi_1{\uparrow^T}$ is irreducible and real valued. Now $X_1=(\eta_1){\uparrow^G}$ and $(\eta_1){\downarrow_{C_0}}=\chi_1$. So $(X_1){\downarrow_T}=(\chi_1){\uparrow^T}$, by Mackey's theorem.

Now from above $\epsilon(X_\lambda)=\epsilon(X_1)$, for all $\lambda\in\Lambda$. Also $\epsilon(X_1)=\epsilon((X_1){\downarrow_T})$, as both are real valued. Finally $\epsilon((X_1){\downarrow_T})=\epsilon_{T/C_0}(\chi_1)$, by the definition. This completes Case 3.
\end{proof}
\end{Theorem}

Finally, we prove the application to ordinary characters as stated in the Introduction:

\begin{proof}[Proof of Theorem \ref{T:symplectic}]
Let $x$ be a weakly real $p$-element of $G$ of maximal order and set $Q:=\langle x\rangle$ and $N:=\op{N}_G(Q)$. Let $\lambda$ be a faithful linear character of $Q$. Then $N_\lambda=\op{C}_N(x)$ and $N_\lambda^*=\op{C}_N^*(x)$. So $N_\lambda^*$ does not split over $N_\lambda$. By Lemma \ref{L:odd_normal} there exists $\chi\in\op{Irr}(N\mid\lambda)$ such that $\epsilon(\chi)=-1$.

Let $\tilde B$ be the $p$-block of $N$ which contains $\chi$ and let $D$ be a defect group of $\tilde B$. Then $Q\subseteq D$ and $\op{N}_G(D)\subseteq N$. In particular $B:=\tilde B^G$ is defined and $B$ has defect group $D$. So $Q=D_i$, $N=N_i$ and $\tilde B=B_i$ for some $i\geq0$, in cyclic defect group notation.

Notice that $\lambda$ is non-trivial. So $D\not\subseteq\op{ker}(\chi)$. This means that $\chi$ is an exceptional character in $B_i$. So all exceptional characters in $B_i$, and hence also in $B$, are symplectic. The number of exceptional characters in $B$ is $\frac{|D|-1}{e}$, where $e$ is the inertial index of $B$. The number of weakly real $p$-conjugacy classes of $G$ is equal to the number of $N$-orbits on $Q^\times$, which equals $\frac{|D_i|-1}{|N_i:C_i|}$. As $|D_i|\leq|D|$ and $e\leq|N_i:C_i|$, we conclude that the number of symplectic irreducible characters of $G$ is not less than the number of weakly real $p$-conjugacy classes of $G$.
\end{proof}

\section{Acknowledgement}

H. Blau and D. Craven alerted us to examples of real $p$-blocks which do not have a real irreducible character. We thank G. Navarro for permission to include his example. We had a number of interesting discussions with R. Gow on F-S indicators.

\end{document}